\documentclass[leqno,12pt,a4paper]{article}
\usepackage[T1]{fontenc}
\usepackage[latin1]{inputenc}
\usepackage{amsmath,amsthm,amssymb}
\usepackage[left=3cm,right=3cm]{geometry}
\usepackage{makeidx}
\usepackage[english]{babel}
\usepackage{amsfonts}
\usepackage{hyperref}

\newtheorem{prop}{Proposition}
\newtheorem{theo}{Theorem}
\newtheorem{lem}{Lemma}
\newtheorem{cor}{Corollary}

\title{Non radial solutions for non homogeneous H\'enon equation}
\author{M. Badiale, G. Cappa}

\begin{document}

%%%%%%%%

\maketitle
%\thispagestyle{empty}\clearpage
%\tableofcontents\clearpage
\textbf{Abstract:} In this paper we study a Hénon-like equation (see equations \eqref{eqhen} below), where the nonlinearity $f(t)$ is not homogeneous (i.e., it is not a power). By minimization on the Nehari manifold, we prove that for large values of the parameter $\alpha$ there is a breaking of symmetry and non radial solutions appears. This holds for sub- and super-critical growth of the nonlinearity $f$.
\\
\rule{10cm}{.1pt} \\
This research was partially supported by the PRIN2012 grant "Aspetti variazionali e perturbativi nei problemi differenziali nonlineari"\\
\rule{10cm}{.1pt} \\
Addresses of authors:\\
Marino Badiale \\
Dipartimento di Matematica, Università di Torino, via Carlo Alberto 10, 10123 Torino, Italy.\\
Email: marino.badiale@unito.it\\
\\
Gianluca Cappa \\
Dipartimento di Matematica e Informatica, Università di Parma, Parco Area delle Scienze 53/A, 43124 Parma, Italy.\\
Email: gianluca.cappa@nemo.unipr.it\\

\section{Introduction}
In this paper we study the following H\'enon-like equation
\begin{equation}
\label{eqhen}
   \left\{
       \begin{array}{ll}
                  -\Delta u=|x|^\alpha f(u) &\text{in}\ \Omega \\
                  u=0 &\text{on}\ \partial\Omega
       \end{array}
   \right.
\end{equation}
where $\Omega$ is the unity ball of $\mathbb{R}^n$ and $n\geq4$. From a well known result of Ni \cite{ni} it derives, assuming suitable hypotheses on $f$, that \eqref{eqhen} has a radial solution. In the case in which $f(t)$ is a power, say $f(t)=|t|^{p-2}t$, the problem is known as H\'enon's equation (see \cite{hen}). A seminal paper of Smets, Su and Willem \cite{smets} showed that, for large $\alpha$'s, there is a breaking of symmetry and a new, non radial, solution appears. After that, much work has been made to study these non radial solutions: multiplicity, shape, asymptotic behavior. To have just an idea of the research on this topic, one can see for example \cite{serr}, \cite{hira} and \cite{bad} for  results about the critical and supercritical cases, \cite{cao} and \cite{long} for the study of the asymptotic behavior of the maximum point and the existence of multi-peak solutions, \cite{lisj} for the uniqueness of the radial solution for $2<p<\frac{2n+2\alpha}{n-2}$, \cite{liya} and \cite{Kolon} for results about above the $p-$Laplacian, \cite{wangya} for the study of H\'enon type system. See also the references in the quoted papers.
\par \noindent At the best of our knowledge, all papers on non radial solutions of H\'enon equation deal with the case in which the non linearity is a power.
\par \noindent In this paper we prove a result of existence of a non radial solution, for large $\alpha$'s, in the case in which $f$ is not a power (but not too different from a power). Borrowing some ideas and some results from \cite{bad}, we also prove existence of non radial solutions for a range of growth of $f$ including supercritical growth. We will find the solutions as minima on the Nehari manifold of the functional usually associated to (\ref{eqhen}). So the main points of the present paper can be summarized as follows: non homogeneous nonlinearity, supercritical growth, Nehari manifold.

\bigskip

To write down our result, we first define $l= n/2$ and $p^*(n)= 2\frac{n+2}{n-2}$ if $n$ is even, $l= [n/2]+1$ and $p^*(n)= 2\frac{[n/2]+2}{[n/2]}$ if $n$ is odd.

We show that such problem admits a radial solution and a non radial solution under the following hypotheses on $f$:
\begin{description}
\item [$(f_1)$] $f$ is a H\"older continuous function (locally), $f(z)\geq0$ $\forall z>0$, $f(z)=\text{o}(z)$ for $z\rightarrow0$; moreover $\lim_{z\rightarrow+\infty}\frac{f(z)}{z}=+\infty$, $f(z)=0$ for all $z\leq 0$;
\item [$(f_2)$] $|f(z)|\leq C(1+|z|)^{p-1}$, where $2<p<p^*(n)$ for all $z$;
\item [$(f_3)$] there exist $q>2$ such that $q F(t)\leq tf(t)$ for all $t\in\mathbb{R}$, where $F(t)= \int_0^t f(s) ds$.
\item [$(f_4)$] there exist $\mu_1,\mu_2>2$ such that for all $t\in [0,1]$ and $v\geq 0$ we have $f(tv)\geq t^{\mu_1-1}f(v)$
  and for all $t\geq 1$ and $v\geq 0$ we have $f(tv)\geq t^{\mu_2-1}g(v)$ where $g(\cdot)$ is a non negative continuous function on $\mathbb{R}$ such that $g(0)=0$ and with
  \[4\frac{\mu_1-\mu_2}{(\mu_1-2)(\mu_2-2)}<n-l.\]
\end{description}

\textbf{Remarks.} Some examples of function that satisfy the hypothesis are $f(t)=t^{p-1}+t^{q-1}$ with $p<q$ and $\mu_1=\mu_2=q$, $f(t)=\frac{t^q}{1+t^{q-p}}$ or $f(t)=\min\{t^{p-1},t^{q-1}\}$ with $p=\mu_1$ and $q=\mu_2$ that satisfy the inequality between the exponent in $(f_4)$.

\medskip

\noindent To state our results we introduce the usual Sobolev space $ H_0^1(\Omega)$ and its subspace $H_{0,rad}^1(\Omega)$ of radial functions, that is

$$H_{0,rad}^1(\Omega)= \left\{ u\in H_0^1(\Omega)/\{0\}:\ u(x)=u(|x|) \right\}.$$

\noindent We the introduce the usual functional associated to problem (\ref{eqhen}), that is

$$I_{\alpha}(u)=\frac{1}{2}\int_{\Omega}|\nabla u|^2 dx-\int_{\Omega}|x|^\alpha F(u)\ dx,$$

 and the Nehari manifolds of the functional on $H_{0,rad}^1(\Omega)$:

$$N_{\alpha,r}=\left\{ u\in H_{0,rad}^1(\Omega)/\{0\}:\ \int_{\Omega}|\nabla u|^2 dx=\int_{\Omega}|x|^\alpha f(u)u\ dx\right\}.$$

\noindent

From the results of \cite{ni} it easily follows that $I_\alpha$ is a well defined $C^1$ functional on $H_{0,rad}^1(\Omega)$, for large $\alpha$'s, that is for $ \frac{2n+2\alpha}{n-2}\geq p^* (n)$. Also the following theorem is a particular consequence, suitable for our purposes, of  the results of \cite{ni} and \cite{will}.
\begin{theo}
Under the hypotheses $(f_1)$, $(f_2)$, $(f_3)$, $(f_4)$, for  $ \frac{2n+2\alpha}{n-2}\geq p^* (n)$, there is $u\in H_{0,rad}^1(\Omega)/\{0\}$, non-negative solution of \eqref{eqhen}, that realizes the minimum on the Nehari manifold $N_{\alpha,r}$ that is:
\[I_{\alpha}(u_\alpha)=m_{\alpha,r}=\min_{v\in N_{\alpha,r}}I_{\alpha}(v).\]
\end{theo}

In this work we proof the following theorem.
\begin{theo}
Under the hypotheses $(f_1)$, $(f_2)$, $(f_3)$ and $(f_4)$ the problem \eqref{eqhen} admits a non radial solution (which is also non negative).
\end{theo}

For future use let us notice that from hypothesis $(f_4)$ it follows that
\[\forall t\in(0,1),\ v>0\ \text{results}\ F(tv)\geq t^{\mu_1}F(v)\]
\[\forall t>1,\ v>0\ \text{results}\ F(tv)\geq t^{\mu_2}G(v)\]
with $G(v):=\int_0^v g(t) dt$.

\noindent The rest  of this paper is devoted to the proof of Theorem 2. As usual in the study of H\'enon equation, to get the proof we first estimate the "radial critical level" $m_{\alpha,r}$, then we estimate other critical levels and we show that, for large $\alpha$'s, they are distinct.

\section{Estimate of $m_{\alpha,r}$}

We have defined
\[m_{\alpha,r}=\inf_{u\in N_{\alpha,r}}I_\alpha(u). \]

We prove the following proposition:

\begin{prop}
\label{proprad}There exists $C>0$ such that
\[m_{\alpha,r}\geq C\ \alpha^\frac{\mu_1+2}{\mu_1-2}\]
for $\alpha\rightarrow+\infty$.
\end{prop}

\noindent In order to prove (\ref{proprad}) we need to introduce some preliminary concept. We remark that, thanks to the results in \cite{ni}, $H_{0,rad}^1(\Omega)\hookrightarrow L^s(\Omega,|x|^\alpha dx)$ for $s<\frac{2(n+\alpha)}{n-2}$ with compact embedding. Define now $a=\frac{1}{2}(n-2)$ and let $u\in H_{0,rad}^1(\Omega)$ be such that
\[\int_{\Omega}|\nabla u|^2 |x|^{-a}dx<+\infty.\]
If $b<a$ then we have $|x|^{-b}<|x|^{-a}$, because $|x|<1$ in $\Omega$, hence

\[\int_{\Omega}|\nabla u|^2 |x|^{-b}dx\leq\int_{\Omega}|\nabla u|^2 |x|^{-a}dx<+\infty.\]
If we extend $u$ setting $u=0$ on $\mathbb{R}^n/\Omega$, we have of course $u\in H_{0,rad}^1(\mathbb{R}^n)$ and
\[\int_{\mathbb{R}^n}|\nabla u|^2 |x|^{-b}dx\leq\int_{\Omega}|\nabla u|^2 |x|^{-a}dx<+\infty.\]

\noindent We now need the following Lemma, which is a particular case of \textbf{Lemma 2.1} of \cite{wang}.

\begin{lem}
Let $\beta \in \mathbb{R}$ be such that $2<n+\beta$. Let $H= H_{0,rad}^1(\mathbb{R}^n,|x|^\beta dx)$ the completion of $C_{0,r}^{\infty}(\mathbb{R}^n )$ under

$$ ||u||_{r,\beta }:= \left(  \int_{\mathbb{R}^n}|x|^\beta |\nabla u|^2 dx \right)^{1/2}.$$

\noindent Then there is $C>0$ such that for all $u \in H$ and for a.e.$x \in \mathbb{R}^n$ it holds

$$  |u(x)| \leq C |x|^{-\frac{n+\beta-2}{2}}.$$

\end{lem}

\bigskip

\noindent Let us now choose $a,b$ as above. We now apply the previous lemma for $\beta=-b$ and for $u\in H_{0,rad}^1(\Omega)$, setting $u=0$ in $\mathbb{R}^n \backslash \Omega$. We get that there exist a constant $C=C_b>0$ such that, for a.e.$x \in \mathbb{R}^n$, it holds
\begin{eqnarray*}
|u(x)|&\leq& C |x|^{-\frac{n-b-2}{2}}\left(\int_{\mathbb{R}^n}|\nabla u|^2 |x|^{-b}dx\right)^{\frac{1}{2}}=\\
      &=& C |x|^{-\frac{n-b-2}{2}}\left(\int_{\Omega}|\nabla u|^2 |x|^{-b}dx\right)^{\frac{1}{2}}\leq\\
      &\leq& C |x|^{-\frac{n-b-2}{2}}\left(\int_{\Omega}|\nabla u|^2 |x|^{-a}dx\right)^{\frac{1}{2}}.
\end{eqnarray*}
\bigskip

\noindent We now prove the following lemma:
\begin{lem}
Let $u\in H_{0,rad}^1(\Omega)$ be such that $\int_{\Omega}|\nabla u|^2 |x|^{-a}dx<+\infty$ and $2<q<\frac{4n}{n-2}$. If $b=n-2-\frac{2n}{q}$ then exist $C=C_b>0$ such that
\[\left(\int_{\Omega}|u|^{q}dx\right)^{\frac{2}{q}}\leq C_b \int_{\Omega}|\nabla u|^2 |x|^{-b}dx\]
\end{lem}
\begin{proof}
First we notice that $q<\frac{4n}{n-2}$ implies $b<\frac{1}{2}(n-2)$, hence $2<n-b$. Also we have $q=\frac{2n}{n-2-b}$. Integrating by parts we get
\begin{eqnarray*}
\int_{\Omega}|u|^{q}dx&=& \omega_n \int_0^1 r^{n-1}|u(r)|^q dr= \\
                      &=&\omega_n \frac{1}{n}r^n|u(r)|^q\bigg|^{r=1}_{r=0}-\frac{q \omega_n}{n}\int_0^1 r^{n}|u(r)|^{q-2}u(r)u'(r)dr.
\end{eqnarray*}

Here $\omega_n$ is the measure of the surface of the unit ball. As $u(1)=0$, we have
\[\frac{1}{n}r^n|u(r)|^q\bigg|^{r=1}_{r=0}=-\lim_{r\rightarrow0}\frac{1}{n}r^n|u(r)|^q\leq0\]
then
\begin{eqnarray*}
\int_{\Omega}|u|^{q}dx&\leq&-\frac{q\omega_n}{n}\int_0^1 r^{n}|u(r)|^{q-2}u(r)u'(r)dr\leq \\
                      &\leq& \frac{2\omega_n}{n-b-2}\int_0^1 r^{n}|u(r)|^{q-1}|u'(r)|dr =\\
                      &=& \frac{2\omega_n}{n-b-2}\int_0^1 r^{n-\frac{n-b-1}{2}}|u(r)|^{q-1}|u'(r)|r^{\frac{n-b-1}{2}}dr \leq\\
                      &\leq& \frac{2\omega_n}{n-b-2}\left( \int_0^1 r^{-b}|u'(r)|^2 r^{n-1}dr\right)^{\frac{1}{2}}\left(\int_0^1
                              r^{n+1+b}|u(r)|^{2(q-1)}dr\right)^{\frac{1}{2}}.
\end{eqnarray*}
Using the previous lemma we get
\begin{eqnarray*}
\int_0^1r^{n+1+b}|u(r)|^{2(q-1)}dr&=&\int_0^1 r^{n-1}|u(r)|^q \, r^{2+b} |u(r)|^{q-2} dr \leq \\
        &\leq& C^{q-2} \int _0^1 r^{n-1}|u(r)|^q \, r^{2+b}\left(r^{-\frac{n-b-2}{2}}\right)^{q-2}dr\\
        &\phantom{=}&\left(\int_{\Omega}|x|^{-b}|\nabla u|^2 dx\right)^{\frac{q-2}{2}}dr\leq\\
        &\leq&C^{q-2} \left(\int_{\Omega}|x|^{-b}|\nabla u|^2 dx\right)^{\frac{q-2}{2}}\int_0^1 r^{n-1}|u(r)|^q dr
\end{eqnarray*}
because
\begin{eqnarray*}
2+b-(q-2)\left(\frac{n-b-2}{2}\right)=0.
\end{eqnarray*}
Recalling that $b< \frac{1}{2}(n-2)$, we have $\frac{1}{n-b-2}< \frac{2}{n-2}$, hence we get
\begin{eqnarray*}
\int_{\Omega}|u|^{q}dx&\leq& C \left(\int_{\Omega}|x|^{-b}|\nabla u|^2 dx\right)^{\frac{q-2}{4}}
                             \left(\int_0^1 r^{n-1}|u(r)|^q dr\right)^{\frac{1}{2}}\\
                &\phantom{=}&\left(\int_0^1 r^{-b}|u'(r)|^2 r^{n-1}dr\right)^{\frac{1}{2}}=\\
                &=&C \left(\int_{\Omega}|x|^{-b}|\nabla u|^2 dx\right)^{\frac{q}{4}}\left(\int_{\Omega}|u|^{q}dx\right)^{\frac{1}{2}},
\end{eqnarray*}
where $C$ is a constant independent from $u$ and $b$, which may change from line to line. We then get
\[\left(\int_{\Omega}|u|^{q}dx\right)^{\frac{1}{2}}\leq C \left(\int_{\Omega}|x|^{-b}|\nabla u|^2 dx\right)^{\frac{q}{4}}\]
that is
\[\left(\int_{\Omega}|u|^{q}dx\right)^{\frac{1}{q}}\leq C \left(\int_{\Omega}|x|^{-b}|\nabla u|^2 dx\right)^{\frac{1}{2}}.\]
\end{proof}

\bigskip

\noindent From this we easily get the following corollary.

\begin{cor}
Let be $u\in H_{0,rad}^1(\Omega)$ such that $\int_{\Omega}|x|^{-a}|\nabla u|^2 dx<+\infty$. If $2<q<\frac{4n}{n-2}$ then exist $C>0$ such that
\[\left(\int_{\Omega}|u|^{q}dx\right)^{\frac{1}{q}}\leq C\left(\int_{\Omega}|x|^{-a}|\nabla u|^2 dx\right)^{\frac{1}{2}}.\]
\end{cor}

\bigskip

\noindent We now introduce the objects we need to work in a Nehari frame. Let us assume  $\alpha >n$ and, as before, $a=\frac{1}{2}(n-2)$. It is easy to check that $p*^(n)< \frac{4n}{n-2}$, so, if $p$ is the exponent in $(f_2 )$, it holds

$$2<p<\frac{4n}{n-2}.$$

\noindent We define

\[ H= \left\{v\in H_{0,rad}^1(\Omega) : \ \  \int_{\Omega}|x|^{-a}|\nabla v|^2 dx < + \infty \right\}, \]

\[ J: H \rightarrow R,   \quad J(v) = \frac{1}{2} \int_{\Omega}|x|^{-a}|\nabla v|^2 dx  - \int_{\Omega}F(v)dx, \]

\[ M=\left\{v\in H_{0,rad}^1(\Omega): \ \ \int_{\Omega}|x|^{-a}|\nabla v|^2 dx=\int_{\Omega}f(v)vdx\right\}\]
and
\[m'=\inf\left\{J(v):\ v\in M\right\}.\]

\noindent Notice that, thanks to $(f_2 )$, $F$ satisfies $|F(t)| \leq c_1 + c_2 |t|^p$ for suitable $c_i >0$. Hence from ${\bf Corollary\ 1}$ and the fact that $2<p<\frac{4n}{n-2}$, it easy to get that the functional $J$ is well defined and $C^1$ on the space $H$.

\noindent Notice also that $M\neq\emptyset$. Indeed it is enough to pick up $\varphi\in C_0^\infty(\Omega/\{0\})$: all the integrals involved in the definition of $M$ are finite and if $\varphi\notin M$ we just rescale it to get $t\varphi\in M$ (for some $t>0$).
Let us now prove that $m'>0$.

\begin{lem}
$m'>0$.
\end{lem}
\begin{proof}

We prove first that $m'\geq 0$. Indeed, from $(f_3 )$ we get that, if $v\in M$, it holds

$$J(v) = \frac{1}{2} \int_{\Omega}|x|^{-a}|\nabla v|^2 dx  - \int_{\Omega}F(v)dx \geq \frac{1}{2} \int_{\Omega}|x|^{-a}|\nabla v|^2 dx -\frac{1}{q} \int_{\Omega}f(v)vdx=$$
$$= \left( \frac{1}{2}-\frac{1}{q} \right) \int_{\Omega}|x|^{-a}|\nabla v|^2 dx \geq 0.$$

To prove $m'>0$, take again $v\in M \subseteq H_{0,rad}^1(\Omega)$. Recalling that $2<p<p^* (n)<\frac{4n}{n-2}$ then, by the previous corollary, we have
\[\left(\int_{\Omega}|v|^{p}dx\right)^{\frac{1}{p}}\leq C\left(\int_{\Omega}|x|^{-a}|\nabla u|^2 dx\right)^{\frac{1}{2}}.\]
Let $\lambda_1$ be the first eigenvalue of the operator $-\Delta$ under zero boundary conditions. Using  hypotheses $(f_1)$ and $(f_2)$ we can choose $C_1>0, $ such that
\[|f(t)t|\leq \frac{1}{2}\lambda_1 t^2+C_1 |t|^{p}\quad\forall t\in\mathbb{R}.\]
Hence we get
\begin{eqnarray*}
\int_{\Omega}|\nabla v|^2 |x|^{-a} dx&=&\int_{\Omega}f(v)v dx\leq \int_{\Omega}|f(v)v| dx\\
                                     &\leq&\frac{\lambda_1}{2}\int_{\Omega}v^2 dx+C_1\int_{\Omega} |v|^{p} dx\\
                                     &\leq&\frac{1}{2}\int_{\Omega}|\nabla v|^2 dx+ C \left(\int_{\Omega}|\nabla v|^2 |x|^{-a}\right)^{\frac{p}{2}}\\
                                     &\leq&\frac{1}{2}\int_{\Omega}|\nabla v|^2 |x|^{-a} dx+ C \left(\int_{\Omega}|\nabla v|^2 |x|^{-a}\right)^{\frac{p}{2}}
\end{eqnarray*}
then
\[\int_{\Omega}|\nabla v|^2 |x|^{-a} dx\geq\left(\frac{1}{C}\right)^{\frac{2}{p-2}}>0\]
for all $v\in M$. As $J(v) \geq \left(\frac{1}{2}-\frac{1}{q}\right)\int_{\Omega}|\nabla v|^2 |x|^{-a} dx$ for all $v\in M$, this implies the lemma.

\end{proof}

\bigskip

\noindent We now apply \textbf{Theorem 1} and we get that there exists a solution $u_\alpha\in H_{0,rad}^1(\Omega)$ to problem (\ref{eqhen}) such that $u_\alpha\in N_{\alpha,r}$ and
\[I_\alpha(u_\alpha)=m_{\alpha,r}=\min_{u\in N_{\alpha,r}}I_\alpha(u).\]
As in \cite{smets}, we define $v_\alpha(x)=v_\alpha(|x|)=u_\alpha(|x|^{\beta})$ where $\beta=\frac{n}{\alpha+n}$, so that $\beta\rightarrow0$ for $\alpha\rightarrow+\infty$. With an obvious change of variables we then obtain

\begin{eqnarray*}
\int_{\Omega}|x|^\alpha f(u_\alpha)u_\alpha dx&=& \omega_n \, \int_0^1 f(u_\alpha(r))u_\alpha(r) r^{\alpha+n-1}dr\\
&=&\omega_n \, \beta\int_0^1 f(v_\alpha(\rho))v_\alpha(\rho) \rho^{\beta(\alpha+n-1)}\rho^{\beta-1}d\rho\\
&=&\omega_n \,   \beta\int_0^1 f(v_\alpha(\rho))v_\alpha(\rho) \rho^{n-1}d\rho=\beta\int_{\Omega}f(v_\alpha(x))v_\alpha(x) dx
\end{eqnarray*}
while
\begin{eqnarray*}
\int_{\Omega}|\nabla u_\alpha|^2 dx &=& \omega_n \, \int_0^1 |u'_\alpha(r)|^2 r^{n-1}dr=\omega_n \, \beta^{-1}\int_0^1 |v'_\alpha(\rho)|^2
                                        \rho^{2-2\beta}\rho^{(n-1)\beta}\rho^{\beta-1}d\rho\\
                                    &=& \omega_n \,  \beta^{-1}\int_0^1 |v'_\alpha(\rho)|^2 \rho^{(2-n)(1-\beta)}\rho^{n-1}d\rho\\
                                    &=& \frac{1}{\beta}\int_{\Omega}|\nabla v_\alpha(x)|^2|x|^{-\gamma} dx
\end{eqnarray*}

\noindent where $\gamma=(n-2)(1-\beta)>0.$

\par \noindent Notice that, for fixed $n$ and $\alpha>n$, we have $\beta <\frac{1}{2}$, so that $\gamma>\frac{1}{2}(n-2)=a$. Then for $|x|<1$ we have $|x|^{-\gamma}>|x|^{-a}$  therefore, as $\int_\Omega |\nabla v_\alpha|^2 |x|^{-\gamma}dx<+\infty$, we obtain
\[\int_\Omega |\nabla v_\alpha|^2 |x|^{-a}dx<+\infty\]

\bigskip

\noindent Let us now define

\[ H_\beta= \left\{v\in H_{0,rad}^1(\Omega) : \ \  \int_{\Omega}|x|^{-\gamma}|\nabla v|^2 dx < + \infty \right\}, \]

\[ J_\beta: H \rightarrow R,   \quad J_\beta(v) = \frac{1}{2} \int_{\Omega}|x|^{-\gamma}|\nabla v|^2 dx  - \int_{\Omega}F(v)dx, \]

\[M_\beta=\left\{v\in H_{0,rad}^1(\Omega)/\{0\}:\ \int_{\Omega}|\nabla v|^2|x|^{-\gamma} dx=\int_{\Omega}f(v)v dx\right\},\]
and
\[m_\beta=\inf\left\{J_\beta(v):\ \ v\in M_\beta\right\}.\]

\noindent As above, thanks to  ${\bf Corollary\ 1}$, $J_\beta$ is well defined and $C^1$ on $H_\beta$, because $\int_{\Omega}|x|^{-\gamma}|\nabla v|^2 dx \leq \int_{\Omega}|x|^{-\gamma}|\nabla v|^2 dx < + \infty $, $|F(t)| \leq c_1 + c_2 |t|^p$ and $2<p<p^* (n) < \frac{4n}{n-2}$.

\noindent Notice that the previous computations imply that $v_{\alpha} \in H_{\beta}$, but in general $v_{\alpha} \notin M_{\beta}$.

\noindent We now prove the following lemma.

\begin{lem}
If $\alpha>n$ and $\beta =\frac{n}{\alpha + n}$ then $m_\beta\geq m'/2$.
\end{lem}

\begin{proof}
If $v\in M_\beta$ then $\int_{\Omega}|\nabla v|^2|x|^{-\gamma} dx<+\infty$ and
\[\int_{\Omega}|\nabla v|^2|x|^{-a} dx\leq\int_{\Omega}|\nabla v|^2|x|^{-\gamma} dx=\int_{\Omega}f(v)v dx.\]
Let us define
\[\psi(t)=\int_{\Omega}|\nabla (tv(x))|^2|x|^{-a} dx-\int_{\Omega}f(tv)tv dx\]
then $\psi(1)\leq0$ and by $(f_1)$ we have that $\psi(t)=t^2\int_{\Omega}|\nabla v|^2|x|^{-a} dx+o(t^2)$ as $t \rightarrow 0$. Hence $\psi(t)\geq0$ for small $t>0$. Then there exists $t_\beta\in(0,1]$ such that $t_\beta v\in M$. Therefore, recalling $(f_3 )$, we get

\begin{eqnarray*}
m'&\leq& J(t_{\beta} v) = \frac{1}{2} \int_{\Omega}|\nabla (t_{\beta} v(x))|^2|x|^{-a} dx - \int_{\Omega} F(t_{\beta}v) dx\\
&\leq&\frac{t_\beta^2}{2} \int_{\Omega}|\nabla ( v(x))|^2|x|^{-a} dx\leq\frac{1}{2} \int_{\Omega}|\nabla (v(x))|^2|x|^{-\gamma} dx\\
&=&\frac{2q}{2+q}\left(\frac{1}{2}-\frac{1}{q}\right)\int_{\Omega}|\nabla ( v(x))|^2|x|^{-\gamma} dx\\
&\leq& \frac{2q}{2+q}\left[\left(\frac{1}{2}-\frac{1}{q}\right)\int_{\Omega}|\nabla ( v(x))|^2|x|^{-\gamma}dx +\int_{\Omega} \left( \frac{1}{q}f(  v)  v - F(v) \right)dx\right]\\
&=&\frac{2q}{2+q}\left[\frac{1}{2}  \int_{\Omega}|\nabla ( v(x))|^2|x|^{-\gamma} dx - \int_{\Omega} F(v) dx\right]=\frac{2q}{2+q}J_\beta(v)<2J_\beta(v).
\end{eqnarray*}

\noindent This holds for every $v\in M_\beta$, so we easily get the thesis.
\end{proof}

\bigskip

\noindent We can now go on with the proof of ${\bf Proposition 1}$.

\begin{proof}[Proof of Proposition 1]
Let us first see that there exist $t_\alpha>0$ such that $t_\alpha v_\alpha \in M_\beta$ that is:
\begin{equation}
 t_\alpha^2\int_{\Omega}|\nabla v_\alpha|^2|x|^{-\gamma} dx=\int_{\Omega}f(t_\alpha v_\alpha)t_\alpha v_\alpha dx.
\label{ta2}
\end{equation}
In fact, let us define
\[\varphi(t)=t^2\int_{\Omega}|\nabla v_\alpha|^2|x|^{-\gamma} dx-\int_{\Omega}f(t v_\alpha)t v_\alpha dx.\]
We have
\begin{eqnarray*}
\varphi(1)&=&\int_{\Omega}|\nabla v_\alpha|^2|x|^{-\gamma} dx-\int_{\Omega}f(v_\alpha)v_\alpha dx=\beta\int_{\Omega}|\nabla u_\alpha|^2
             dx-\frac{1}{\beta}\int_{\Omega}|x|^\alpha f(u_\alpha)u_\alpha dx \\
          &=&\frac{1}{\beta}\left[\beta^2\int_{\Omega}|\nabla u_\alpha|^2 dx-\int_{\Omega}|x|^\alpha f(u_\alpha)u_\alpha dx\right]\\
          &=&\frac{1}{\beta}\left[\beta^2\int_{\Omega}|\nabla u_\alpha|^2 dx-\int_{\Omega}|\nabla u_\alpha|^2 dx\right]\\
          &=&\frac{1}{\beta}(\beta^2-1)\int_{\Omega}|\nabla u_\alpha|^2 dx<0.
\end{eqnarray*}
On the other hand, for $t\rightarrow0^+$, it is easy to get that for each $\varepsilon>0$ there exists $C_\varepsilon>0$ such that
\[\left|\int_{\Omega}f(t v_\alpha)t v_\alpha dx\right|\leq\varepsilon t^2\int_{\Omega}v_\alpha^2 dx+C_\varepsilon t^p\int_{\Omega}v_\alpha^p dx,\]
hence
\[\int_{\Omega}f(t v_\alpha)t v_\alpha dx=o(t^2),\quad t\rightarrow0^+ .\]
So we get
\[\varphi(t)=t^2\int_{\Omega}|\nabla v_\alpha|^2|x|^{-\gamma} dx+o(t^2),\quad t\rightarrow0^+,\]
therefore $\varphi(t)\geq0$ for $t\rightarrow0^+$. It is easy to deduce that there exist $t_\alpha\in(0,1)$ such that $\varphi(t_\alpha)=0$ i.e. \eqref{ta2}.
\\
From the hypothesis ($f_4$) we have that
\[t_\alpha^2\int_{\Omega}|\nabla v_\alpha|^2|x|^{-\gamma} dx=\int_{\Omega}f(t_\alpha v_\alpha)t_\alpha v_\alpha dx \geq t_\alpha^{\mu_1}\int_{\Omega}f(v_\alpha)v_\alpha dx\]
that is
\[t_\alpha^{{\mu_1}-2}\leq\frac{\int_{\Omega}|\nabla v_\alpha|^2|x|^{-\gamma} dx}{\int_{\Omega}f(v_\alpha)v_\alpha dx}
               =\frac{\beta\int_{\Omega}|\nabla u_\alpha|^2 dx}{\frac{1}{\beta}\int_{\Omega}f(u_\alpha)u_\alpha dx}=\beta^2,  \]

\noindent because $u_\alpha\in N_{\alpha,r}$. Hence

\[t_\alpha\leq\beta^{\frac{2}{{\mu_1}-2}}.\]

Therefore
\begin{eqnarray*}
\frac{m'}{2}&\leq& m_\beta\leq J_\beta(t_\alpha v_\alpha)=\frac{1}{2}t^2_\alpha\int_{\Omega}|\nabla v_\alpha|^2|x|^{-\gamma}dx-\int_{\Omega}F(t_\alpha v_\alpha)dx\\
  &\leq& \frac{1}{2}t^2_\alpha\int_{\Omega}|\nabla v_\alpha|^2|x|^{-\gamma}dx=C\left(\frac{1}{2}-\frac{1}{q}\right)t_\alpha^2\int_{\Omega}|\nabla
         v_\alpha|^2|x|^{-\gamma}dx\quad\left(\text{with}\ C=\frac{q}{q-2}>0\right)\\
  &\leq& C\left(\frac{1}{2}-\frac{1}{q}\right)t_\alpha^2\int_{\Omega}|\nabla v_\alpha|^2|x|^{-\gamma}dx+C\ t_\alpha^{\mu_1} \int_{\Omega}\left(\frac{1}{q}
         f(v_\alpha)v_\alpha-F(v_\alpha)\right)dx\quad\big(\text{by}\ (f_3)\big)\\
  &\leq& \beta^\frac{4}{{\mu_1}-2}\beta\ C\left(\frac{1}{2}-\frac{1}{q}\right)\int_{\Omega}|\nabla u_\alpha|^2dx+\beta^\frac{2{\mu_1}}{{\mu_1}-2}\beta^{-1}\
         C\int_{\Omega}|x|^\alpha\left(\frac{1}{q}f(u_\alpha)u_\alpha-F(u_\alpha)\right)dx\\
  &=& \beta^\frac{{\mu_1}+2}{{\mu_1}-2}\ C\left[\left(\frac{1}{2}-\frac{1}{q}\right)\int_{\Omega}|\nabla u_\alpha|^2dx+\int_{\Omega}|x|^\alpha\left(\frac{1}{q}
         f(u_\alpha)u_\alpha-F(u_\alpha)\right)dx\right]\\
  &=& \beta^\frac{{\mu_1}+2}{{\mu_1}-2}\ C\left[\frac{1}{2}\int_{\Omega}|\nabla u_\alpha|^2dx-\int_{\Omega}|x|^\alpha F(u_\alpha)dx\right]\quad(\text{since }u_\alpha\in N_{\alpha,r})\\
  &=& \beta^\frac{{\mu_1}+2}{{\mu_1}-2}\ C\ m_{\alpha,r}.
\end{eqnarray*}
So
\[m_{\alpha,r}\geq m'\frac{1}{C}\left(\frac{1}{\beta}\right)^\frac{{\mu_1}+2}{{\mu_1}-2}=C'\left(\frac{\alpha+n}{n}\right)^\frac{{\mu_1}+2}{{\mu_1}-2}\]
that is
\[m_{\alpha,r}\geq C''\ \alpha^\frac{{\mu_1}+2}{{\mu_1}-2}.\]
\end{proof}

\section{Other critical levels and their estimates}

In this section we follow \cite{bad}. Recall that we have defined $l= n/2$ id $n$ is even and $l= [n/2]$ if $n$ is odd. Let us now  set $x=(y,z)\in\mathbb{R}^{l}\times\mathbb{R}^{n-l}$ and define

$$H_l=\{u\in H_0^1(\Omega):\ u(y,z)=u(|y|,|z|)\}, \quad N_{\alpha,l}=\{u\in H_l:\ I'_\alpha(u)u=0\}, $$
$$m_\alpha^l=\inf_{u\in N_{\alpha,l}}I_\alpha(u).$$

\noindent Thanks to the results of \cite{bad} (see in particular \textbf{Corollary 2.3}) we have that, for $2<p<p^*(n)$ and $\alpha >n+2$, $I_{\alpha }$ is well defined and $C^1$ in $H_l$.

\bigskip

\noindent As first thing we prove that $m_\alpha^l > 0$.

\begin{prop}
$m_\alpha^l > 0$.

\end{prop}
\begin{proof}

For  $v\in N_{\alpha,l}$ we have that
\[\int_\Omega |\nabla v|^2 dx =\int_\Omega |x|^\alpha f(v)v dx\]
and from hypothesis $(f_1)$,  $(f_2)$, for every $\eta>0$ there exists $C_\eta>0$ such that
\[|f(z)z|\leq\eta z^2+C_\eta z^p,\ \ \ \forall z\in\mathbb{R}.\]
Then applying the \textbf{Corollary 2.3} of \cite{bad} we obtain
\begin{eqnarray*}
\int_\Omega |\nabla v|^2 dx&\leq& \eta\int_\Omega |x|^\alpha v^2 dx+C_\eta\int_\Omega |x|^\alpha |v|^p dx\\
                           &\leq& D_1 \eta \int_\Omega |\nabla v|^2 dx+C_\eta D_2\left(\int_\Omega |\nabla v|^2 dx\right)^p,
\end{eqnarray*}
where the constants $D_1, D_2$ are independent from $\eta$. Let us now set  $\|v\|^2:=\int_\Omega |\nabla v|^2 dx$. We have
\[\|v\|^2\leq C \eta\|v\|^2+D_\eta\|v\|^p, \]

\noindent with $C$ independent from $\eta$. We can then choose $\eta$ such that $1-\eta C>0$, so that
\[\|v\|\geq\left(\frac{1-\eta C}{D_\eta}\right)^\frac{1}{p-2}>0.\]

By $(f_3 )$ we also get, for $v\in N_{\alpha,l}$,

$$I_\alpha (u) = \frac{1}{2}\int_\Omega |\nabla v|^2 dx - \int_\Omega |x|^\alpha F(v) dx \geq \frac{1}{2} \int_\Omega |\nabla v|^2 dx - \frac{1}{q} \int_\Omega |x|^\alpha f(v)v dx \geq $$
$$ \left( \frac{1}{2} -  \frac{1}{q} \right) \int_\Omega |\nabla v|^2 dx \geq C >0. $$

\end{proof}

\noindent We now prove the following proposition:

\begin{prop}
There exists $C>0$ such that
\[m_\alpha^l\leq  C\alpha^{\frac{{\mu_2}+2}{{\mu_2}-2}-n+l}\]
for $\alpha\rightarrow+\infty$.
\end{prop}

\begin{proof}
We consider
\[D=\{(s,t)\in\mathbb{R}^2:\ s,t\geq0,\ 0\leq s^2+t^2\leq 1\}.\]

\noindent For $u\in H_l$ we have

\[\int_{\Omega} |x|^\alpha F(u) dx=C\int_{D}(s^2+t^2)^\frac{\alpha}{2}F(u(s,t))s^{l-1}t^{n-l-1}dsdt,\]
\[\int_{\Omega} |\nabla u|^2 dx=C\int_{D} |\nabla u(s,t)|^2 s^{l-1}t^{n-l-1}dsdt.\]

We study $I_\alpha $ on $H_l$ using polar coordinates, that is we set $s=\rho\cos\theta$, $t=\rho\sin\theta$ and define
\[A=\{(\rho,\theta)\in\mathbb{R}^2:\ 0\leq\rho<1,\ 0\leq\theta\leq2\pi\}\]
and
\[v(\rho,\theta)=u(\rho\cos\theta,\rho\sin\theta).\]

\noindent Hence we get
\[\int_{D}(s^2+t^2)^\frac{\alpha}{2}F(u(s,t))s^{l-1}t^{n-l-1}dsdt=\int_A F(v(\rho,\theta))\rho^{\alpha+n-1}H(\theta)d\rho d\theta\]
\[\int_{D} |\nabla u(s,t)|^2 s^{l-1}t^{n-l-1}dsdt=\int_A \left(v_\rho(\rho,\theta)^2+\frac{1}{\rho^2}v_\theta(\rho,\theta)^2\right)\rho^{n-1}H(\theta)d\rho d\theta\]
where $H(\theta)=(\sin(\theta))^{n-l-1}(\cos(\theta))^{l-1}$.
Therefore on $H_l$ we have
\begin{eqnarray*}
I_\alpha(u)&=&\frac{1}{2} \int_{\Omega} |\nabla u|^2 dx-\int_{\Omega} |x|^\alpha F(u) dx\\
&=&\frac{1}{2}\left(\int_A \left(v_\rho^2+\frac{1}{\rho^2}v_\theta^2\right)\rho^{n-1}H(\theta)d\rho d\theta-\int_A F(v)\rho^{\alpha+n-1}H(\theta)d\rho d\theta\right).
\end{eqnarray*}

Now we introduce
\[\tilde{A}=\left(\frac{1}{4},\frac{3}{4}\right)\times(\theta_1,\theta_2)\]
with $0<\theta_1<\theta_2<\pi/2$, and we consider anon negative function $\psi\in C_0^\infty(\tilde{A})\backslash \{ 0 \}$. For $\varepsilon>0$ we define
\[v^\varepsilon(\rho,\theta)=\psi\left(\rho^\frac{1}{\varepsilon},\frac{\theta}{\varepsilon}\right).\]
We get that $v^\varepsilon\in C_0^\infty(\tilde{A_\varepsilon})$ where
\[\tilde{A_\varepsilon}=\left\{(\rho,\theta)\in\mathbb{R}^2:\ \left(\frac{1}{4}\right)^\varepsilon<\rho<\left(\frac{3}{4}\right)^\varepsilon,\ \varepsilon\theta_1<\theta<\varepsilon\theta_2\right\}.\]
We want evaluate $I_\alpha$ at $u^\varepsilon $, the functions defined by
\[u^\varepsilon(x)=u^\varepsilon(|y|,|z|)= u^\varepsilon(\rho\cos\theta,\rho\sin\theta)=v^\varepsilon(\rho,\theta);\]
obviously it holds $u^\varepsilon \in C_0^\infty(\Omega)\cap H_l(\Omega)$.
Now we define
\[\varepsilon=\frac{n}{\alpha+n}\]
so that $\varepsilon\rightarrow0$ when $\alpha$ goes to infinity. We compute
\[\int_\Omega |\nabla u^\varepsilon|^2 dx=\int_{\tilde{A}}\left(\psi_1^2+\frac{1}{r^2}\psi_2^2\right)r^{(\varepsilon-1)(n-2)}r^{n-1}H(\varepsilon\varphi) dr d\varphi,\ \ \ (\varepsilon-1)(n-2)<0\]
and
\[\int_\Omega |x|^\alpha F(u^\varepsilon) dx=\varepsilon^2\int_{\tilde{A}} F(\psi) r^{n-1}H(\varepsilon\varphi) dr d\varphi\]
Now we prove that if $\varepsilon=\frac{n}{\alpha+n}$ is small enough then there exists $t_\varepsilon>1$ such that $t_\varepsilon u^\varepsilon\in N_{\alpha,l}$.

For this we put
\[h(t):=t^2\int_\Omega |\nabla u^\varepsilon|^2 dx-\int_\Omega f(t u^\varepsilon)t u^\varepsilon dx.\]
From $(f_1)$ we easily derives $h(t)\rightarrow-\infty$ for $t\rightarrow+\infty$ while
\begin{eqnarray*}
h(1)&=&\int_\Omega |\nabla u^\varepsilon|^2 dx-\int_\Omega f( u^\varepsilon) u^\varepsilon dx\\
    &=&C\left(\int_{\tilde{A}}\left(\psi_1^2+\frac{1}{r^2}\psi_2^2\right)r^{(\varepsilon-1)(n-2)}r^{n-1}H(\varepsilon\varphi) dr d\varphi-\varepsilon^2\int_{\tilde{A}}f(\psi)\psi r^{n-1}H(\varepsilon\varphi) dr d\varphi \right)\\
    &\geq&C\varepsilon^{n-l-1}\left(\int_{\tilde{A}}\left(\psi_1^2+\frac{1}{r^2}\psi_2^2\right)r^{n-1}drd\varphi-\varepsilon^{2}\int_{\tilde{A}}f(\psi)\psi r^{n-1} dr d\varphi\right)>0\\
\end{eqnarray*}
if $\varepsilon$ is small enough; then there exists $t_\varepsilon>1$ such that $h(t_\varepsilon)=0$, that is $t_\varepsilon u^\varepsilon\in N_{\alpha,l}$.

Then we have
\begin{eqnarray*}
t_\varepsilon^2\int_\Omega |\nabla u^\varepsilon|^2 dx&=&\int_\Omega |x|^\alpha f(t_\varepsilon u^\varepsilon)t_\varepsilon u^\varepsilon dx\geq t_\varepsilon^{\mu_2}\int_\Omega |x|^\alpha g(u^\varepsilon)u^\varepsilon dx\\
      &=& C t_\varepsilon^{\mu_2} \varepsilon^2\int_{\tilde{A}} g(\psi)\psi r^{n-1}H(\varepsilon\varphi) dr d\varphi
\end{eqnarray*}
that is
\begin{eqnarray*}
t_\varepsilon^{\mu_2-2}&\leq& C\frac{\int_\Omega |\nabla u^\varepsilon|^2 dx}{\varepsilon^2\int_{\tilde{A}} g(\psi)\psi r^{n-1}H(\varepsilon\varphi) dr d\varphi}\\
 &\leq& C \frac{\int_{\tilde{A}}\left(\psi_1^2+\frac{1}{r^2}\psi_2^2\right)r^{(\varepsilon-1)(n-2)}r^{n-1}H(\varepsilon\varphi) dr d\varphi}{\varepsilon^{2+n-l-1}\int_{\tilde{A}} g(\psi)\psi r^{n-1} dr d\varphi}\\
 &\leq& C \frac{\varepsilon^{n-l-1}\int_{\tilde{A}}\left(\psi_1^2+\frac{1}{r^2}\psi_2^2\right)r^{n-1}dr d\varphi}{\varepsilon^{2+n-l-1}\int_{\tilde{A}} g(\psi)\psi r^{n-1} dr d\varphi}\leq C\varepsilon^{-2}.
\end{eqnarray*}

Then

$$m_\alpha^l \leq I_\alpha(t_\varepsilon u^\varepsilon)=\frac{1}{2}t_\varepsilon^2\int_\Omega |\nabla u^\varepsilon|^2 dx-\int_\Omega |x|^\alpha F(t_\varepsilon u^\varepsilon) dx$$
    $$\leq \frac{1}{2}t_\varepsilon^2\int_\Omega |\nabla u^\varepsilon|^2 dx-t_\varepsilon^{\mu_2}\int_\Omega |x|^\alpha G(u^\varepsilon) dx$$
 $$\leq \frac{1}{2}t_\varepsilon^2\int_\Omega |\nabla u^\varepsilon|^2 dx-t_\varepsilon^{\mu_2}\int_\Omega |x|^\alpha G(u^\varepsilon) dx$$

$$\leq C_1\varepsilon^{-\frac{4}{\mu_2-2}}\int_{\tilde{A}}\left(\psi_1^2+\frac{1}{r^2}\psi_2^2\right)r^{(\varepsilon-1)(n-2)}r^{n-1}H(\varepsilon\varphi) dr d\varphi-$$
          $$\phantom{\leq}-C_2\varepsilon^{-\frac{2\mu_2}{\mu_2-2}}\varepsilon^2\int_{\tilde{A}} G(\psi) r^{n-1}H(\varepsilon\varphi) dr d\varphi$$
          $$\leq C_3 \varepsilon^{-\frac{4}{\mu_2-2}+n-l-1}+ C_4 \varepsilon^{-\frac{2\mu_2}{\mu_2-2}+2+n-l-1}$$
          $$= C\varepsilon^{-\frac{\mu_2+2}{\mu_2-2}+n-l}=C\left(\frac{n}{\alpha+n}\right)^{-\frac{\mu_2+2}{\mu_2-2}+n-l}\leq C\alpha^{\frac{\mu_2+2}{\mu_2-2}-n+l}$$

\end{proof}

\noindent We can now conclude the proof of Theorem 2.

\begin{proof}[Proof of Theorem 2]
Thanks to Theorem 1 we know that $m_{\alpha, r}$ is a minimum, that is, there is a radial $u$ assuming it, and this $u$ is a solution to (\ref{eqhen}). The compactness results in \cite{bad} (see in particular \textbf{Corollary 2.3}) imply that also $m_{\alpha}^l $ is assumed by a solution $v$ of (\ref{eqhen}). Both solutions are non trivial, because $m_{\alpha, r}\not= 0 \not= m_{\alpha}^l $ and non negative, because we assume $f(t)=0$ for $t\leq 0$. So it is enough to prove that these solutions are different, and a way to see this is to prove that the critical levels are different, that is $m_{\alpha, r}\not= m_{\alpha}^l $, at least for large $\alpha$'s. To see this, we notice that, from the hypothesis $(f_4)$, we have
\[\frac{{\mu_2}+2}{{\mu_2}-2}+l-n<\frac{{\mu_1}+2}{{\mu_1}-2}.\]
We easily deduce that, for $\alpha\rightarrow+\infty$
\[\alpha^{\frac{{\mu_2}+2}{{\mu_2}-2}+1-n}<\alpha^{\frac{{\mu_1}+2}{{\mu_1}-2}}\]
and finally
\[m_{\alpha}^l <m_{\alpha,r}.\]

\noindent We have then obtained a non radial non trivial solution of (\ref{eqhen}), and the theorem is proved.

\end{proof}

%\addcontentsline{toc}{section}{References}


\begin{thebibliography}{12}
\bibitem[1]{smets} SMETS, D. - SU, J. - WILLEM, M. - \textit{Non radial ground state solution fore the H\'enon equation.} Comm Contemp Math, 4; 467-480 (2002)
\bibitem[2]{liya} LI, Z. - YANG, Z. - \textit{Bifurcation method for solving multiple positive solutions to boundary value problem of p-H\'enon equation on the unit disk.} Appl Math Mech Engl Ed. 31; 511-520 (2010)
\bibitem[3]{lisj} LI, SJ. - PENG, SJ. - \textit{Asymptotic behavior on the H\'enon equation with supercritical exponent.} Science in China Series A: Mathematics, Vol. 52, No. 10; 2185-2194 (Oct., 2009)
\bibitem[4]{hira} HIRANO, N. - \textit{Existence of positive solutions for the H\'enon equation involving critical Sobolev terms.} J. Differential Equations 247, 1311-1333 (2009)
\bibitem[5]{cao} CAO, D. - PENG, S. - YAN, S. - \textit{Asymptotic behaviour of ground state solutions for the H\'enon equation.} IMA Journal of Applied Mathematics  74, 468-480 (2009)
\bibitem[6]{ni} NI, W. M. -  \textit{A Nonlinear Dirichlet Problem on the Unit Ball and its Applications.} Indiana Univ. Math. J 31; 801-807 (1982)
\bibitem[7]{long} LONG W. - YANG J. - \textit{Existence and asymptotic behavior of solutions for H\'enon type equations.} Opuscola Matematica, Vol. 31, No. 3; 411-424 (2011)
\bibitem[8]{will} WILLEM, M. - \textit{Minimax Theorems.} Progress in Nonlinear Differential Equations and Their Applications, Vol. 24 (Birkh\"auser, Boston, 1996)
\bibitem[9]{bad} BADIALE, M. - SERRA, E. \textit{Multiplicity Results for the Supercritical H\'enon Equation.} Advanced Nonlinear Studies 4; 543-467 (2004)
\bibitem[10]{hen} H\'ENON, M. - \textit{Numerical experiments on the stability of spherical stellar systems.} Astronomy and Astrophysics 24;  229-238, (1973)
\bibitem[11]{serr} SERRA, E. \textit{Non radial positive solutions for the H\'enon equation with critical growth.} Calc. Var. Partial Differential Equations 23; 301-326 (2005)
\bibitem[12]{wang} SU, J. - WANG, Z. \textit{Sobolev type embedding and quasilinear elliptic equations with radial potentials.} Journal of Differential Equations 250, 223-243 (2011)
\bibitem[13]{wangya} WANG, Y. - YANG, J. \textit{Asymptotic behavior of ground state solution for H\'enon type system.} Electronic Journal of Differential Equations, Vol. 2010, No. 116; 1-14 (2010)
\bibitem[14]{Kolon} KOLONITSKII, S. B. - NAZAROV, A. I. \textit{Multiplicity of solutions to the Dirichlet problem for generalized H\'enon equation.} Journal of Mathematical Sciences, Vol. 144, No. 6, 4624-4644 (2007)

\end{thebibliography}
\end{document}